\documentclass[12pt,a4paper]{article}
\usepackage{amsmath,amssymb,amsfonts}
\usepackage{amsthm}
\usepackage{mathtools}
\usepackage{marvosym}
\usepackage{authblk}
\usepackage{enumerate}
\usepackage[english]{babel}
\usepackage{enumitem}
\usepackage{todonotes}
\usepackage{indentfirst}
\usepackage{geometry}
\usepackage{graphicx}
\usepackage{tabularx}
\usepackage{tocbibind}
\usepackage{tikz}
\usepackage{array}
\usepackage{caption}
\usepackage{setspace}
\usepackage{hyperref}
\usepackage{lipsum}
\usepackage{listings}
\usepackage{multicol,multirow}
\usepackage{xcolor,xspace}
\usepackage[english]{babel}
\usepackage[export]{adjustbox}

\newtheorem{theorem}{Theorem}[section]
\newtheorem{corollary}{Corollary}[theorem]
\newtheorem{lemma}[theorem]{Lemma}
\newtheorem{prop}[theorem]{Proposition}
\newtheorem{result}[theorem]{Result}
\newtheorem{example}[theorem]{Example}
\theoremstyle{definition}
\newtheorem{definition}{Definition}[section]
\theoremstyle{remark}

\title{\textbf{On the strong geodeticity in the corona type product of graphs}}
\author[1]{Bishal Sonar\thanks{Email: bsonarnits@gmail.com}}
\author[2]{Satyam Guragain\thanks{Email: shatym17@gmail.com}}
\author[3]{Ravi Srivastava\thanks{Corresponding author Email: ravi@nitsikkim.ac.in}}
\affil[1,2,3]{Department of Mathematics, National Institute of Technology Sikkim, South Sikkim 737139, India}
\date{}
\begin{document}
\parskip1ex
\parindent0pt
\maketitle

\begin{abstract}
    \noindent The paper focuses on studying strong geodetic sets and numbers in the context of corona-type products of graphs. Our primary focus is on three variations of the corona products: the generalized corona, generalized edge corona, and generalized neighborhood corona products. A strong geodetic set is a minimal subset of vertices that covers all vertices in the graph through unique geodesics connecting pairs from this subset. We obtain the strong geodetic set and number of the corona-type product graph using the strong 2-geodetic set and strong 2-geodetic number of the initial arbitrary graphs. We analyze how the structural properties of these corona products affect the strong geodetic number, providing new insights into geodetic coverage and the relationships between graph compositions. This work contributes to expanding research on the geodetic parameters of product graphs.
\end{abstract}


\section{Introduction}
    Graph theory, known for its complex combinatorial aspects, has become an essential branch of mathematical study due to its broad applicability in computer science, network theory, and biology. A key focus in this area is the exploration of geodetic sets and geodesic paths, which involve identifying the smallest collection of vertices that encompass all the shortest paths between any two vertices in a graph. This minimal set is known as the geodetic set, and its size is referred to as the geodetic number of the graph. Harary and Buckley~\cite{buckley1990f} made significant contributions to the foundational principles of understanding geodesic paths and their relationship with the structure of graphs. Their work serves as the basis for numerous research problems \cite{manuel2022geodesic, foucaud2022monitoring, manuel2023geodesic,zhang2008exact, harary1993geodetic,pelayo2013geodesic,atici2002computational}. Building on these initial works, Manuel et al.~\cite{manuel2018strong} introduced the concept of the strong geodetic set, which imposes a stricter criterion by requiring each pair of vertices in the set to be connected by a unique geodesic that includes all other vertices in the graph. The size of this set is called the strong geodetic number. This concept has undergone extensive scrutiny, producing significant findings across various graph classes and their composite structures.

    Expanding on this research, we examine the analysis of strong geodetic sets and numbers as they relate to different variations of the corona product of graphs. The corona product is a graph operation that combines two distinct graphs by connecting the vertices of one graph to those of another, resulting in a new graph with properties inherited from both original graphs. Our focus is on three variations of the corona product: the general corona product, the generalized edge corona product, and the generalized neighborhood corona product.

    In the general corona product~\cite{frucht1970corona}, every vertex of a principal graph is linked to all vertices of a corresponding subgraph. In the generalized edge corona product, the endpoints of edges in the principal graph are connected to the vertices of the subgraphs. Meanwhile, the generalized neighborhood corona product establishes connections between the vertices of subgraphs and the neighborhood of a respective vertex in the base graph. Each variant presents a distinct structure that significantly impacts the geodetic attributes of the resulting graph, including the strong geodetic number.

    This paper aims to examine the influence of these different products on the strong geodetic number, revealing new understanding and generalizations within graph theory. By determining the strong geodetic numbers relevant to each of these product types, we hope to expand the current knowledge base concerning strong geodetic sets and strong geodetic numbers, offering a comprehensive exploration of how the structure of products affects geodesic coverage. All the graphs considered in this paper are simple and connected. 

    \subsection{Preliminaries}
        Let $G=(V, E)$ be a connected graph, where $V$ denotes the set of vertices, and $E$ denotes the edges with $|V|\geq2$. The cardinality of $V$ is called the order of the graph $G$, and the cardinality of $E$ is called the size of $G$. We denote the degree of a vertex $u$ in $G$ by $d(u)$. A vertex with degree one is called a pendent vertex. $N(u)$ is a subset of $V(G)$ containing all the vertices adjacent to $u$. $E(u)$ is a subset of $E(G)$, denotes the set of edges having one of its vertex $u$. A path is a walk with non-repeating vertices; the shortest path connecting two vertices $u$ and $v$ in $G$ is called the geodesic or an isometric path. We denote it by $geodesic(u,v)$, and its distance is denoted by $d(u,v)$. The length of any longest geodesic from a vertex $u\in V$ is the eccentricity $e(u)$ of $u$, and the maximum of the eccentricities of all of the vertices in $G$ is the diameter of $G$ denoted as $diam(G)$ \cite{harary1990distance}. A graph is geodetic if a unique isometric path connects every pair of vertices in $G$ \cite{harary1990distance}. If a vertex $u$ in a graph $G$ lies on a $geodesic(v,w)$ geodesic in $G$ for any two vertices $v,w\in V$, then the pair $(v,w)$ is said to geodominate(cover) $u$. If every vertex of $G$ is covered by the geodesic joining any two vertices in $\eta$, then $\eta$ is considered a geodesic cover of $G$. The geodetic cover of the least cardinality is called the geodetic basis of $G$, and the cardinality of the basis set is called the geodetic number denoted by $g(G)$ \cite{harary1990distance}. A set $\eta_{S_g}\subseteq V$ is called a strong geodetic set of $G$ if all the vertices in $V\smallsetminus \eta_{S_g}$ are covered using geodesic that is fixed between the elements of $\eta_{S_g}$, in a manner that every pair of vertices in $\eta_{S_g}$ is assigned a unique geodesic. If we denote $\tilde{I}[(s,t)]$ as the geodesic that is fixed between two vertices $s$ and $t$ of $\eta_{Sg}$ and $\tilde{I}[\eta_{Sg}]=\{\tilde{I}[(s,t)]:s,t\in \eta_{Sg}\}$, then $\eta_{Sg}$ is called a strong geodetic set if $V(\tilde{I}[\eta_{Sg}])=V(G)$. The minimum order of such a set is called the strong geodetic number denoted by $Sg(G)$, and any such set of least order is called the strong geodetic basis \cite{manuel2018strong}. Two vertices are antipodal if they are farthest from each other.

        \begin{definition}\cite{laali2016spectra}[Generalized corona product]
            Given simple graphs $G, H_1,\ldots, H_n$, where $n=|V(G)|$, the generalized corona, denoted by $G~\Tilde{\circ}\overset{n}{\underset{i=1}{ \Lambda}} H_i$, is the graph obtained by taking one copy of graphs $G,H_1,\ldots,H_n$ and joining the $i^\text{th}$ vertex of $G$ to every vertex of $H_i$. 
        \end{definition}

        \begin{definition}\cite{abdolhosseinzadeh2017generalized}[Generalized edge corona product]
            Given simple graphs $G, H_1,\ldots, H_n$, where $n=|V(G)|$, the generalized edge corona, denoted by $G~\Tilde{\diamond}\overset{n}{\underset{i=1}{ \Lambda}} H_i$, is the graph obtained by taking one copy of graphs $G,H_1,\ldots,H_n$ and joining the two end vertices of $i^\text{th}$ edge of $G$ to every vertex of $H_i$.
        \end{definition}

        \begin{definition}\cite{setiawan2021bounds}[Generalized neighborhood corona product]
            Given simple graphs $G, H_1,\ldots, H_n$, where $n=|V(G)|$, the generalized neighborhood corona, denoted by $G~\Tilde{\star}\overset{n}{\underset{i=1}{ \Lambda}} H_i$, is the graph obtained by taking one copy of graphs $G,H_1,\ldots,H_n$ and joining each vertex in $H_i$ to the neighborhood of the $i^\text{th}$ vertex of $G$.
        \end{definition}

\section{Strong geodetic number}

    \begin{definition}[2-geodesic]\label{Def 1}
        A $geodesic(x,y)$ is called 2-geodesic if $1\leq d(x,y)\leq2$.
    \end{definition}

    \begin{definition}[2-geodetic cover]
        If every vertex of $G$ is covered by the 2-geodesic joining any two vertices in $\eta'\subseteq V(G)$, then $\eta'$ is called the 2-geodetic cover of $G$.
    \end{definition}

    \begin{definition}[Strong 2-geodetic set]
        A set $\eta'_{S_g}\subseteq V(G)$ is called a strong geodetic set of $G$ if all the vertices in $V(G)\smallsetminus \eta'_{S_g}$ are covered using 2-geodesic that are fixed between the elements of $\eta'_{S_g}$ in a manner that every pair of vertices in $\eta'_{S_g}$ is assigned a unique 2-geodesic.
    \end{definition}

    \begin{definition}[Strong 2-geodetic basis]
        The strong 2-geodetic cover of the least cardinality of a graph $G$ is called the strong 2-geodetic basis, denoted by $\eta'_{Sg}(G)$.
    \end{definition}

    \begin{definition}[Strong 2-geodetic number]
        The cardinality of the strong 2-geodetic basis of a graph $G$ is called the strong 2-geodetic number, denoted by $Sg(G)$.
    \end{definition}

%
%
        \begin{lemma}\label{Lemma 0}
            For any Graph $G$, the pendent vertex of $G$ cannot be covered by any two vertices in $G$, i.e., pendent vertices always belong to the strong geodetic set of $G$.
        \end{lemma}

        \begin{proof}
            For a vertex to be covered by a pair of vertices, the degree of that vertex must be at least $2$, but the degree of a pendent vertex is $1$. So, a pendent vertex cannot be covered by any other pair of vertices. Hence, the pendent vertices always belong to the strong geodetic set of the graph.
        \end{proof}

%
%
    \begin{result}\label{Result 1}
        For a graph $G$, $Sg(G)=Sg'(G)$ if $0\leq diam(G)\leq2$.
    \end{result}
    \begin{proof}
        Let for any graph $G$ $diam(G)\leq2$, then for any $u,v\in V(G)$, $1\le d(u,v)\le2$. And by Definition \ref{Def 1}, each geodesic path with vertices in $\eta_{Sg}(G)$ are the 2-geodesic path. So $Sg(G)=Sg'(G)$.
    \end{proof}
    The converse of the above theorem is not true. It can be easily verified by the example given in Figure \ref{fig:1}. Here, $Sg(G)=Sg'(G)$, but $diam(G)=3$.
    \begin{figure}
        \centering
        \begin{tikzpicture}
        \node (A) at (-3, 0) [circle, fill=black, inner sep=2pt, label=above:a] {};
        \node (B) at (-1, 0) [circle, fill=black, inner sep=2pt, label=above:b] {};
        \node (C) at (1, 0) [circle, fill=black, inner sep=2pt, label=above:c] {};
        \node (D) at (3, 0) [circle, fill=black, inner sep=2pt, label=above:d] {};
        \node (E) at (-1, -2) [circle, fill=black, inner sep=2pt, label=below:e] {};
        \node (F) at (1, -2) [circle, fill=black, inner sep=2pt, label=below:f] {};

        \draw (A) -- (B);
        \draw (B) -- (C);
        \draw (C) -- (D);
        \draw (B) -- (E);
        \draw (C) -- (F);
        
\end{tikzpicture}
        \caption{$Sg(G)=Sg'(G)=\{a,d,e,f\}$ but $diam(G)=3$}
        \label{fig:1}
    \end{figure}
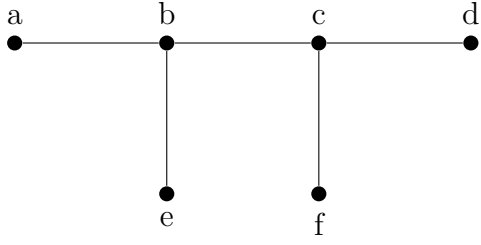
    
%
%
    \subsection{Generalized corona product}
        Let $G$ be a graph of order $n$ and $H_1,H_2,\ldots,H_n$ be $n$ graphs of order $t_1,t_2,\ldots,t_n$. We denote the set of vertices by $V(G)=\{u_1,u_2,\ldots,u_n\}$, and $V(H_i)=\{v^i_1,v^i_2,\ldots,v^i_{t_i}\}$. The generalized corona is given by $G~\Tilde{\circ}\overset{n}{\underset{i=1}{ \Lambda}} H_i$, and we have the following: 
%
%
    \begin{lemma}\label{Lemma 1}
        For $u_i,u_j\in V(G)$, then $u_i,u_j$ can be covered by any geodesic $(v_p^i,v_q^j)$ \big[$v_p^i\in V(H_i)$, $v_q^j\in V(H_j)$\big], for any $1\leq i<j\leq n$ and for any $1\leq p\leq t_i$, and $1\leq q\leq t_j$.
    \end{lemma}

    \begin{proof}
        Each $v_p^i\in V(H_i),~ p=1,2,\ldots,t_i,$ is adjacent to $u_i$ in $V(G)$ and each $v_q^j\in V(H_j),~ q=1,2,\ldots,t_j,$ is adjacent to $u_j$ in $V(G)$. So the $geodesic(v_p^i,v_q^j)$ will cover $u_i$ and $u_j$. As $u_i,u_j, ~1\leq i<j\leq n$ and $v_p^i$ and $v_q^j$ are arbitrary, we have the lemma.
    \end{proof}
%
%
    \begin{lemma}\label{Lemma 2}
        If $v^i_p,v^i_q$ in $V(H_i)$ are non-adjacent then, $2=geodesic(v_p^i,v_q^i) (\text{in } G~\Tilde{\circ}\overset{n}{\underset{i=1}{ \Lambda}} H_i)\leq geodesic(v^i_p,v^i_q)(\text{in } H_i)$.
    \end{lemma}

    \begin{proof}
        Each vertex of $H_i$ is adjacent to the $i^\text{th}$ vertex of $G$. Suppose $v^i_p,v^i_q$ be two vertices in $H_i$ such that geodesic$(v^i_p,v^i_q)=m\geq 2$. Then in the generalized Corona product $v_p^i$ and $v_q^i$ are both adjacent to $u_i$, $i=1,2,\ldots,n$. So, the length of the minimum path(geodesic) connecting $v_p^i$ and $v_q^i$ is $2$.
        Hence, $2=geodesic(v_p^i,v_q^i) (\text{in } G~\Tilde{\circ}\overset{n}{\underset{i=1}{ \Lambda}} H_i)\leq geodesic(v^i_p,v^i_q)(\text{in } H_i)$.
    \end{proof}
%
%
    \begin{lemma}\label{Lemma 3}
        For a fixed strong 2-geodetic cover $\eta'_{Sg}(H_k)$ of a graph $H_k;~k=1,2,\ldots,n$. Any vertex $v_p^k\in \eta'_{Sg}(H_k)$ cannot be covered by a vertex in $\eta'_{Sg}(H_k)$ and a vertex in $V(G~\Tilde{\circ}\overset{n}{\underset{i=1}{ \Lambda}} H_i)\smallsetminus V(H_k)$.
    \end{lemma}
    
    \begin{proof}
        Clearly, $u_k$ is adjacent to each $v_i^k\in V(H_k)$. Suppose $v_i^k$ adjacent to $v_p^k$ and there exists a geodesic $(v_i^k,v_p^k,u_k)$ that covers $v_p^k$, but clearly $(v_i^k,u_k)$ is a path of shortest length from $v_i^k$ and $u_k$, this is a contradiction to the fact that $(v_i^k,v_p^k,u_k)$ is a geodesic path. Hence we cannot cover a vertex in $\eta'_{Sg}(H_k)~k=1,2,\ldots,n$ by a vertex in $\eta'_{Sg}(G_2^k)$ and a vertex in $V(G~\Tilde{\circ}\overset{n}{\underset{i=1}{ \Lambda}} H_i)\smallsetminus V(H_k)$.
    \end{proof}
%
%
    \begin{theorem}\label{Theorem 1}
        Let $G,H_1,H_2,\ldots,H_n$, where $|V(G)|=n$ be $N+1$ graphs. If the strong 2-geodetic basis of $H_i$ is $\eta'_{Sg}(H_i)=\{b_1^i,b_2^i,\ldots,b_{s_i}^i\},~i=1,2,\ldots,n$, then the strong geodetic basis of the generalized corona product $\eta_{Sg}(G~\Tilde{\circ}\overset{n}{\underset{i=1}{ \Lambda}} H_i)=\{b_1^1,b_2^1,\ldots,b_{s_1}^1,b_1^2,b_2^2,\ldots,b_{s_n-1}^n,b_{s_n}^n\}$ and the strong geodetic number $Sg(G~\Tilde{\circ}\overset{n}{\underset{i=1}{ \Lambda}} H_i)=\overset{n}{\underset{i=1}{\sum}}s_i$.
    \end{theorem}

    \begin{proof}
        Let $G$ be a graph of order $n$ and $H_1,H_2,\ldots,H_n$ be graphs(not necessarily isomorphic) of order $t_1,t_2,\ldots,t_n$. Let $V(H_i)=\{v^i_1,v^i_2,\ldots,v^i_{t_i}\}$ and $\eta'_{Sg}(H_i)=\{b_1^i,b_2^i,\ldots,b_{s_1}^i\}$ be the fixed 2-geodetic basis of $H_i,~i=1,2,\ldots,n$.
        Now by Lemma \ref{Lemma 1} each vertex of $G$ can be covered by the vertices in $\overset{n}{\underset{i=1}{\cup}}\eta'_{Sg}(H_i)$. Also by Lemma \ref{Lemma 3} we have each $\eta'_{Sg}(H_i),~i=1,2,\ldots,n$, remains unchanged. So, $\overset{n}{\underset{i=1}{\cup}}\eta'_{Sg}(H_i)=\{b_1^1,b_2^1,\ldots,b_{s_1}^1,b_1^2,b_2^2,\ldots,b_{s_n-1}^n,b_{s_n}^n\}$ covers $G~\Tilde{\circ}\overset{n}{\underset{i=1}{ \Lambda}} H_i$. Now, if we remove any vertex from $\overset{n}{\underset{i=1}{\cup}}\eta'_{Sg}(H_i)$, it will not cover $G~\Tilde{\circ}\overset{n}{\underset{i=1}{ \Lambda}} H_i$ as each $\eta'_{Sg}(H_i)$ is a 2-geodetic basis of $H_i$ and by Lemme \ref{Lemma 3} vertices in $H_i$ cannot be covered by a vertex in $\eta'_{Sg}(H_i)$ and a vertex in $V(G~\Tilde{\circ}\overset{n}{\underset{i=1}{ \Lambda}} H_i)\smallsetminus V(H_k),$ for all $i=1,2,\ldots,n$. Hence, $\overset{n}{\underset{i=1}{\cup}}\eta'_{Sg}(H_i)$ is a geodetic basis of $G~\Tilde{\circ}\overset{n}{\underset{i=1}{ \Lambda}} H_i$.\\ Next, as each $\eta'_{Sg}(H_i)$ are disjoint we have the strong geodetic number, $$Sg(G~\Tilde{\circ}\overset{n}{\underset{i=1}{ \Lambda}} H_i)=\overset{n}{\underset{i=1}{\sum}}s_i.$$
    \end{proof}
%
%
    \begin{corollary} \label{Corollary 1}
        If each $H_i,~i=1,2,\ldots,n$ are isomorphic denoted by $H$ of order $m$, and let the strong 2-geodetic set of H be $\eta'_{Sg}(H)=\{b_1,b_2,\ldots,b_s\}$. Then the generalized corona product reduces to the corona product, and the strong geodetic set is given by $\eta(G\circ H)=\{b^1_1,b^1_2,\ldots,b^1_s,b^2_1,\ldots,b^n_{s-1},b^n_s\}$. And the strong geodetic number $Sg(G\circ H)=n\cdot {Sg'}(H)$.
    \end{corollary}

    \begin{proof}
        The proof follows directly from Theorem \ref{Theorem 1}.
    \end{proof}
%
%
    \begin{prop}\label{Proposition 1}
        $v_p^i,~1\leq p\leq t_i$ and $v_q^j,~1\leq q\leq t_j$ are antipodal in $G~\Tilde{\circ}\overset{n}{\underset{i=1}{ \Lambda}} H_i$ if and only if $u_i$ and $u_j$ are antipodal in $G$.
    \end{prop}

    \begin{proof}
        For each $i=1,2,\ldots,n$, $v_p^i$ is adjacent to $u_i$, $1\leq p\leq t_i$. So, if the diameter of $G$ $diam(G)=k$, then $diam(G~\Tilde{\circ}\overset{n}{\underset{i=1}{ \Lambda}} H_i)=k+2$.\\ Suppose, $u_i$ and $u_j$ are antipodal. Then $d(v_p^i,v_q^j)=d(v_p^i,u_i)+d(u_i,u_j)+d(u_j,v_q^j)=k+2=diam(G~\Tilde{\circ}\overset{n}{\underset{i=1}{ \Lambda}} H_i)$. Hence, $v_p^i,~1\leq p\leq s_i$ and $v_q^j,~1\leq q\leq s_j$ are antipodal in $G~\Tilde{\circ}\overset{n}{\underset{i=1}{ \Lambda}} H_i$.\\
        Conversely, let $v_p^i$ and $v_q^j$ be antipodal, then $d(v_p^i,v_q^j)=diam(G~\Tilde{\circ}\overset{n}{\underset{i=1}{ \Lambda}} H_i)$. Then $d(u_i,u_j)=diam(G~\Tilde{\circ}\overset{n}{\underset{i=1}{ \Lambda}} H_i)-2=diam(G)$. Hence, $u_i$ and $u_j$ are antipodal.
    \end{proof}
%
%
    \begin{example}
        Let $G=C_3$, and $H_1=P_2$, $H_2=K_4, H_3=C_5$.
    \end{example}
    \textbf{Solution:} Here, the possible strong 2-geodetic set of $H_1,H_2,H_3$ are given by $\eta'_{Sg}(H_1)=\{a_1,a_2\};$ $\eta'_{Sg}(H_2)=\{b_1,b_2,b_3,b_4\};$ $\eta'_{Sg}(H_3)=\{c_1,c_3,c_4\}$. So, by Theorem \ref{Theorem 1} we have the strong geodetic set $\eta_{Sg}(G~\Tilde{\circ}\overset{3}{\underset{i=1}{ \Lambda}} H_i)=\{a_1,a_2,b_1,b_2,b_3,b_4,c_1,c_3,c_4\}$. And the Strong geodetic number $Sg(G~\Tilde{\circ}\overset{3}{\underset{i=1}{ \Lambda}} H_i)=2+4+3=9$.

    From \cite{SHARMA201714} we know the corona graphs generated by the seed graph $G$ of order $n$ is given by $G^{(m+1)}=G^{(m)}\circ G$, where $G^{(0)}=G$. The number of vertices in $G^{(m)}$ is $|V^{(m)}|=n(n+1)^m$. Let, $G$ be a graph with 2-geodetic set $\eta'_{Sg}(G)=\{b_1,b_2,\ldots,b_s\}$, and the 2-geodetic set of the $i^\text{th}$ copy of $G$ is given by $\eta'_{Sg}(G_i)=\{b^i_1,b^i_2,\ldots,b^i_s\}~i=1,2,\ldots,n$. Then we have the following theorem for the corona graphs.
%
%
    \begin{theorem}\label{Theorem 1a}
        The strong geodetic set of the corona graphs $G^{(m+1)}$ is given by $\eta_{Sg}(G^{(m+1)})=\{b^1_1,b^1_2,\ldots,b^1_s,b^2_1,b^2_2,\ldots,b^{n(n+1)^m}_{s-1},b^{n(n+1)^m}_{s}\}$, and the strong geodetic number is given by\\ \ $Sg(G^{(m+1)})=s\cdot n(n+1)^m$.
    \end{theorem}
    \begin{proof}
        By Theorem \ref{Theorem 1} it is clear that the strong geodetic set of $G^{(m+1)}$ will not contain any vertex from the graph $G^{(m)}$ and will only have vertices from the strong 2-geodetic set of the $n(n+1)^m$ copies of $G$ in the corona product of $G^{(m)}\circ G$. Hence the strong geodetic set of the corona graphs $G^{(m+1)}$ is given by $\eta_{Sg}(G^{(m+1)})=\{b^1_1,b^1_2,\ldots,b^1_s,b^2_1,b^2_2,\ldots,b^{n(n+1)^m}_{s-1},b^{n(n+1)^m}_{s}\}$, and the strong geodetic number is given by $Sg(G^{(m+1)})=|\eta_{Sg}(G^{(m+1)})|=s\cdot n(n+1)^m$.
    \end{proof}
    
%
%
    \subsection{Generalized edge corona}
        Let $G$ be a graph of order $n$ and $H_1,H_2,\ldots,H_n$ be $n$ graphs of order $t_1,t_2,\ldots,t_n$. We denote the set of vertices by $V(G)=\{u_1,u_2,\ldots,u_n\}$, and $V(H_i)=\{v^i_1,v^i_2,\ldots,v^i_{t_i}\}$. The generalized edge corona is given by $G~\Tilde{\diamond}\overset{n}{\underset{i=1}{ \Lambda}} H_i$, and we have the following:
%
%
        \begin{lemma}\label{Lemma 4}
            Each $u_i\in V(G)$ can be covered by vertices in $\overset{n}{\underset{i=1}{\cup}}\eta'_{Sg}(H_i)$, if any of the following conditions holds:
            \begin{enumerate}
                \item If $d(u_i)\ge2$, then for some $e_a,e_b\in E(u_i)(a\neq b)$, $(v^a_p,v^b_q)~ \big[1\le p\le s_a,~ 1\le q\le s_b$, $v_p^a\in \eta'_{Sg}(H_a),~ v_q^b\in \eta'_{Sg}(H_a)\big]$ will cover $u_i$.
                \item If $d(u_i)=1$(i.e., $u_i$ is a pendent vertex and $e_a$(say) is the edge connecting $u_i$), then a fixed 2-$geodesic(v_p^a,v_q^a)$, $[v_p^a,v_q^a\in\eta'_{Sg}(H_a),~1\le p<q\le s_a]$ will cover $u_i$ if and only if there exists at least three geodesic $(v_p^a,v),~ (v_p^a,v_q^a)$ and $(v,v_q^a)$, $[v\in \eta'_{Sg}(H_a)]$ of length 2 covering only a particular vertex $(u\in V(H_a))$.  
            \end{enumerate}
        \end{lemma}

        \begin{proof} For arbitrary $u_i\in V(G)$;
            \begin{enumerate}
                \item When $d(u_i)\ge 2$, then at least two $e_a,e_b\in E(u_i)$, then each $v_p^a,~ 1\le p\le s_a$ and $v_q^b,~ 1\le q\le s_b$ are adjacent to $u_i$. The path $(v_p^a,u_i,v_q^b)$ covers $u_i$, and as it is one of the geodesic paths covering $u_i$, we can fix this geodesic to cover $u_i$.
                
                \item For, $v_p^q,v_q^a\in \eta'_{Sg}(H_a)$, let the $geodesic(v_p^a,v_q^a)$ of length 2 covers $u_i$, then $v_p^a$ and $v_q^a$ are adjacent to $u_i$. Now, if $u\in V(H_a)$ is covered by a $geodesic(v_p^a,v_q^a)$ of length 2, which is fixed in $\eta'_{Sg}(H_a)$ to cover $u$, then the $geodesic(v_p^a,v_q^a)$ cannot cover $u_i$. Hence, there exists at least two more 2-geodesic $(v_p^a,v),$ and $(v,v_q^a)~ v\in \eta'_{Sg}(H_a)$ such that one of the remaining two 2-geodesic can be fixed to cover $u\in V(H_a)$.\\
                Conversely, if there exists at least three geodesic $(v_p^a,v),~ (v_p^a,v_q^a)$ and $(v,v_q^a)$, $[v\in \eta'_{Sg}(H_a)]$ of length 2 covering a vertex $u\in V(H_a)$, then without loss of generality we can fix any of the geodesic $(v_p^a,v)$ or $(v_q^a,v),~ v\in\eta'_{Sg}(H_a)$ to cover $u$ and the 2-$geodesic(v_p^a,v_q^a)$ can be fixed to cover $u_i$.\\
                Lastly, if there exists a unique 2-geodesic to cover each $u$ in $V(H_a)$, then $u_i$ cannot be covered by vertices in $\eta'_{Sg}(H_a)$.
            \end{enumerate}
        \end{proof}
%
%
        \begin{lemma}\label{Lemma 5}
            For a fixed strong 2-geodetic cover $\eta'_{Sg}(H_k)$ of a graph $H_k;~k=1,2,\ldots,n$. Any vertex $v_p^k\in \eta'_{Sg}(H_k)$ cannot be covered by a vertex in $\eta'_{Sg}(H_k)$ and a vertex in $V(G~\Tilde{\diamond}\overset{n}{\underset{i=1}{ \Lambda}} H_i)\smallsetminus V(H_k)$.
        \end{lemma}

        \begin{proof}
            The proof is similar to Lemma \ref{Lemma 3}.
        \end{proof}
%
%
        \begin{theorem}\label{Theorem 2}
            Let $G,H_1,H_2,\ldots,H_n$, where $|V(G)|=n$ be $N+1$ graphs. If the strong 2-geodetic basis of $H_i$ is $\eta'_{Sg}(H_i)=\{b_1^i,b_2^i,\ldots,b_{s_i}^i\},~i=1,2,\ldots,n$, then we have the following results for the strong geodetic basis and the strong geodetic number of the generalized edge corona product:
            \begin{enumerate}
                \item If $G$ is a graphs with no pendent vertices, then $\eta_{Sg}(G~\Tilde{\diamond}\overset{n}{\underset{i=1}{ \Lambda}} H_i)=\{b_1^1,b_2^1,\ldots,b_{s_1}^1,b_1^2,b_2^2,$\\$\ldots,b_{s_n-1}^n,b_{s_n}^n\}$ and $Sg(G~\Tilde{\diamond}\overset{n}{\underset{i=1}{ \Lambda}} H_i)=\overset{n}{\underset{i=1}{\sum}}s_i$.
                \item If $G$ is a graph with $p$ many pendent vertices $\{a_1,a_2,\ldots,a_p\}$. Let $e_{a_i}$ be the edge corresponding to the pendent vertex $a_i$ and $H_{a_i}$ be the corresponding graph. Let $A$ be the subset of $V(G)$ such that for each vertex of $A$, the corresponding graph $H_{a_i}$ has the property that every vertex of that graph is not covered by more than one geodesic of length 2. Then, $\eta_{Sg}(G~\Tilde{\diamond}\overset{n}{\underset{i=1}{ \Lambda}} H_i)=\{b_1^1,b_2^1,\ldots,b_{s_1}^1,b_1^2,b_2^2,\ldots,b_{s_n-1}^n,b_{s_n}^n\}\cup A$ and $Sg(G~\Tilde{\diamond}\overset{n}{\underset{i=1}{ \Lambda}} H_i)=\overset{n}{\underset{i=1}{\sum}}s_i+|A|$.
            \end{enumerate}
            \end{theorem}
            
            \begin{proof}
                Let $G$ be a graph of order $n$ and $H_1,H_2,\ldots,H_n$ be graphs(not necessarily isomorphic) of order $t_1,t_2,\ldots,t_n$. Let $V(H_i)=\{v^i_1,v^i_2,\ldots,v^i_{t_i}\}$ and $\eta'_{Sg}(H_i)=\{b_1^i,b_2^i,\ldots,b_{s_1}^i\}$ be the fixed 2-geodetic basis of $H_i,~i=1,2,\ldots,n$.
                \begin{enumerate}
                    \item Let $G$ be a graph with no pendent vertex, then by Lemma \ref{Lemma 4} each vertex of $G$ is covered by the vertices in $\overset{n}{\underset{i=1}{\cup}}\eta'_{Sg}(H_i)$. Also by Lemma \ref{Lemma 5} we have each $\eta'_{Sg}(H_i),~i=1,2,\ldots,n$, remains unchanged. So, $\overset{n}{\underset{i=1}{\cup}}\eta'_{Sg}(H_i)=\{b_1^1,b_2^1,\ldots,b_{s_1}^1,b_1^2,b_2^2,\ldots,b_{s_n-1}^n,b_{s_n}^n\}$ covers $G~\Tilde{\diamond}\overset{n}{\underset{i=1}{ \Lambda}} H_i$. Now, if we remove any vertex from $\overset{n}{\underset{i=1}{\cup}}\eta'_{Sg}(H_i)$, it will not cover $G~\Tilde{\diamond}\overset{n}{\underset{i=1}{ \Lambda}} H_i$ as each $\eta'_{Sg}(H_i)$ is a 2-geodetic basis of $H_i$ and by Lemme \ref{Lemma 5} vertices in $H_i$ cannot be covered by a vertex in $\eta'_{Sg}(H_i)$ and a vertex in $V(G~\Tilde{\diamond}\overset{n}{\underset{i=1}{ \Lambda}} H_i)\smallsetminus V(H_k),$ for all $i=1,2,\ldots,n$. Hence, $\overset{n}{\underset{i=1}{\cup}}\eta'_{Sg}(H_i)$ is the geodetic basis of $G~\Tilde{\diamond}\overset{n}{\underset{i=1}{ \Lambda}} H_i$. Next, as each $\eta'_{Sg}(H_i)$ are disjoint we have the strong geodetic number, $Sg(G~\Tilde{\diamond}\overset{n}{\underset{i=1}{ \Lambda}} H_i)=\overset{n}{\underset{i=1}{\sum}}s_i$.

                    \item Let $G$ be a graph with $p$ many pendent vertices $\{a_1,a_2,\ldots,a_p\}$. Then, by Lemma \ref{Lemma 5} and Point 1, all the vertices in $V(G)$ are covered except those in the set $A$. So, by Lemma \ref{Lemma 0} these vertices belong to the strong geodetic set $\eta_{Sg}(G~\Tilde{\diamond}\overset{n}{\underset{i=1}{ \Lambda}} H_i)$. Now, following similar arguments as before, we have the strong geodetic set $\eta_{Sg}(G~\Tilde{\diamond}\overset{n}{\underset{i=1}{ \Lambda}} H_i)=\{b_1^1,b_2^1,\ldots,b_{s_1}^1,b_1^2,b_2^2,\ldots,b_{s_n-1}^n,b_{s_n}^n\}\cup A$. Lastly, the strong geodetic number, \\$Sg(G~\Tilde{\diamond}\overset{n}{\underset{i=1}{ \Lambda}} H_i)=\overset{n}{\underset{i=1}{\sum}}s_i+|A|$.
                \end{enumerate}
                
            \end{proof}

%
%
            \begin{corollary}\label{Corollary 2}
                Let $G$ be a graph of order $n$, with $p$ many pendent vertices, $p=1,2,\ldots,n$ denoted by $\{a_1,a_2,\ldots,a_p\}$. If each $H_i,~i=1,2,\ldots,n$ are isomorphic denoted by $H$ of order $m$, and let the strong 2-geodetic set of H be $\eta'_{Sg}(H)=\{b_1,b_2,\ldots,b_s\}$. Then the generalized edge corona product reduces to the edge corona product, and the strong geodetic set is given by $\eta(G\diamond H)=\{b^1_1,b^1_2,\ldots,b^1_s,b^2_1,\ldots,b^n_{s-1},b^n_s\}\cup A$. And the strong geodetic number $Sg(G\diamond H)=n\cdot {Sg'}(H)+|A|$.
            \end{corollary}

            \begin{proof}
                The proof follows directly from Theorem \ref{Theorem 2}.
            \end{proof}

%
%

    \subsection{Generalized neighborhood corona}
        Let $G$ be a graph of order $n$ and $H_1,H_2,\ldots,H_n$ be $n$ graphs of order $t_1,t_2,\ldots,t_n$, then the generalized neighborhood corona is given by $G~\Tilde{\star}\overset{n}{\underset{i=1}{ \Lambda}} H_i$, and we have the following:
%
%
        \begin{lemma}\label{Lemma 6}
            If $u_i\in V(G),~i=1,2,\ldots,n$, then $u_i$ can be covered by any geodesic$(v_p^i,v_q^j),[v_p^i\in V(H_i) \text{ and } v_q^j\in V(H_j)]$, where $u_j\in N(u_i)=\{u_p: u_i \text{ adjacent to } u_p\}$.
        \end{lemma}

        \begin{proof}
            Each $v_p^i\in V(H_i)$ is adjacent to $u_j\in N(u_i)$, and each $v_q^j\in V(H_j)$ is adjacent to $u_i$. So the $geodesic(v_p^i,v_q^j)$ covers $u_i$, for all $i=1,2,\ldots,n$.
        \end{proof}
%
%
        \begin{lemma} \label{Lemma 7}
            For a fixed strong 2-geodetic cover $\eta'_{Sg}(H_k)$ of a graph $H_k;~k=1,2,\ldots,n$. Any vertex $v_p^k\in \eta'_{Sg}(H_k)$ cannot be covered by a vertex in $\eta'_{Sg}(H_k)$ and a vertex in $V(G~\Tilde{\star}\overset{n}{\underset{i=1}{ \Lambda}} H_i)\smallsetminus V(H_k)$.
        \end{lemma}

        \begin{proof}
            The proof is similar to Lemma \ref{Lemma 3}.
        \end{proof}
%
%
        \begin{theorem}\label{Theorem 3}
            Let $G,H_1,H_2,\ldots,H_n$, where $|V(G)|=n$ be $N+1$ graphs. If the strong 2-geodetic basis of $H_i$ is $\eta'_{Sg}(H_i)=\{b_1^i,b_2^i,\ldots,b_{s_i}^i\},~i=1,2,\ldots,n$, then the strong geodetic basis of the generalized neighborhood corona product $\eta_{Sg}(G~\Tilde{\star}\overset{n}{\underset{i=1}{ \Lambda}} H_i)=\{b_1^1,b_2^1,\ldots,b_{s_1}^1,b_1^2,b_2^2,\ldots,b_{s_n-1}^n,$\\$b_{s_n}^n\}$ and the strong geodetic number $Sg(G~\Tilde{\star}\overset{n}{\underset{i=1}{ \Lambda}} H_i)=\overset{n}{\underset{i=1}{\sum}}s_i$.
        \end{theorem}

        \begin{proof}
            Let $G$ be a graph of order $n$ and $H_1,H_2,\ldots,H_n$ be graphs(not necessarily isomorphic) of order $t_1,t_2,\ldots,t_n$. Let $V(H_i)=\{v^i_1,v^i_2,\ldots,v^i_{t_i}\}$ and $\eta'_{Sg}(H_i)=\{b_1^i,b_2^i,\ldots,b_{s_1}^i\}$ be the fixed 2-geodetic basis of $H_i,~i=1,2,\ldots,n$.
            Now by Lemma \ref{Lemma 6} each vertex of $G$ is covered by the vertices in $\overset{n}{\underset{i=1}{\cup}}\eta'_{Sg}(H_i)$. Also by Lemma \ref{Lemma 7} we have each $\eta'_{Sg}(H_i),~i=1,2,\ldots,n$, remains unchanged. So, $\overset{n}{\underset{i=1}{\cup}}\eta'_{Sg}(H_i)=\{b_1^1,b_2^1,\ldots,b_{s_1}^1,b_1^2,b_2^2,\ldots,b_{s_n-1}^n,b_{s_n}^n\}$ covers vertices in $G~\Tilde{\star}\overset{n}{\underset{i=1}{ \Lambda}} H_i$. Now, if we remove any vertex from $\overset{n}{\underset{i=1}{\cup}}\eta'_{Sg}(H_i)$, it will not cover $G~\Tilde{\star}\overset{n}{\underset{i=1}{ \Lambda}} H_i$ as each $\eta'_{Sg}(H_i)$ is a 2-geodetic basis of $H_i$ and by Lemme \ref{Lemma 7} vertices in $H_i$ cannot be covered by a vertex in $\eta'_{Sg}(H_i)$ and a vertex in $V(G~\Tilde{\star}\overset{n}{\underset{i=1}{ \Lambda}} H_i)\smallsetminus V(H_k),$ for all $i=1,2,\ldots,n$. Hence, $\overset{n}{\underset{i=1}{\cup}}\eta'_{Sg}(H_i)$ is the geodetic basis of $G~\Tilde{\star}\overset{n}{\underset{i=1}{ \Lambda}} H_i$.\\ Next, as each $\eta'_{Sg}(H_i)$ are disjoint we have the strong geodetic number, $$Sg(G~\Tilde{\star}\overset{n}{\underset{i=1}{ \Lambda}} H_i)=\overset{n}{\underset{i=1}{\sum}}s_i.$$
        \end{proof}
%
%
    \begin{corollary} \label{Corollary 3}
        If each $H_i,~i=1,2,\ldots,n$ are isomorphic, denoted by $H$ of order $m$. The generalized neighborhood corona product reduces to the neighborhood corona product, and the strong geodetic set is given by $\eta(G\circ H)=\{b^1_1,b^1_2,\ldots,b^1_m,b^2_1,\ldots,b^n_{s-1},b^n_s\}$. And the strong geodetic number $Sg(G\star H)=n\cdot Sg'(H)$.
    \end{corollary}

    \begin{proof}
        The proof follows directly from Theorem \ref{Theorem 3}.
    \end{proof}

%
%
    \begin{prop}\label{Proposition 3}
        $v_p^i,~1\leq p\leq t_i$ and $v_q^j,~1\leq q\leq t_j$ are antipodal in $G~\Tilde{\star}\overset{n}{\underset{i=1}{ \Lambda}} H_i$ if $u_i$ and $u_j$ are antipodal in $G$ and has no common neighbor. If $u_i$ and $u_j$ are antipodal in $G$ and has a common neighbor say $u_k$, then the vertices $v_p^i,~1\leq p\leq t_i$ and $v_r^k,~1\leq r\leq t_k$, and also the vertices $v_q^j,~1\leq q\leq t_j$ and $v_r^k,~1\leq r\leq t_k$ are antipodal.
    \end{prop}

    \begin{proof}
        Let $u_i$ and $u_j$ be antipodal vertices in $G$ and let $(u_i,a_1,a_2,\ldots,a_k,u_j)$ be the geodesic path, $l\in \mathbb{N}$.
        Also, each $v_p^i $ for $1\leq p\leq t_i$ is adjacent to all the neighborhoods of $u_i$, and each $v_q^j $ for $1\leq q\leq t_j$ is adjacent to all the neighborhoods of $u_j$.\\
        Case: I $u_i$ and $u_j$ do not have a common neighbor.\\
        In this case $(v_p^i,a_1,a_2,\ldots,a_k,v_q^j)$ is a geodesic path, so $d(v_p^i,v_q^j)=d(u_i,u_j)$ and hence $v_p^i$ and $v_q^j$ are antipodal for $1\leq p\leq t_i$ and $1\leq q\leq t_j$.\\
        Case: II $u_i$ and $u_j$ has a common neighborhood $u_k$(say).\\
        Then each $v_r^k$, for $1\leq r\leq t_k$ is adjacent to $u_i$ and $u_j$, and
        \begin{equation*}
            d(v_p^i,v_r^k)=\begin{cases}
                3 &\text{if } u_i \text{ and } u_k \text{ has no common neighbor}\\
                2 &\text{if } u_i \text{ and } u_k \text{ has a common neighbor.}
            \end{cases}
        \end{equation*}
        In both the scenarios $v_p^i$ and $v_r^k$ are antipodal for $1\leq p\leq t_i$ and $1\leq r\leq t_l$.\\
        Similarly, we have $v_q^j$ and $v_r^k$ are antipodal for $1\leq q\leq t_j$ and $1\leq r\leq t_l$.
    \end{proof}

\bibliographystyle{abbrv} 
\bibliography{main.bib}

\end{document}